\documentclass[12pt, reqno]{amsart}
\usepackage{mathrsfs}
\usepackage{amsfonts}
\usepackage[centertags]{amsmath}
\usepackage{amssymb}
\usepackage{amsthm}
\usepackage{graphicx}
\usepackage{caption}
\usepackage{hyperref}
\usepackage{enumerate}
\usepackage[textwidth=16cm, hmarginratio=1:1]{geometry}
\usepackage{rotating}
\usepackage[all]{xy}
\usepackage[titletoc, title]{appendix}
   
\usepackage{color}

\newtheorem{theorem}{Theorem}[section]
\newtheorem{corollary}[theorem]{Corollary}
\newtheorem{lemma}[theorem]{Lemma}
\newtheorem{proposition}[theorem]{Proposition}

\newtheorem{remark}[theorem]{Remark}

\numberwithin{equation}{section}

\DeclareMathOperator*{\id}{id}

\DeclareMathOperator*{\End}{End}

\DeclareMathOperator*{\Hom}{{Hom}}

\DeclareMathOperator*{\spn}{{Span}}

\def\sl{\mathfrak{sl}}

\def\ra{\rightarrow}

\def\a{\alpha}

\def\k{\kappa}

\def\e{\epsilon}

\def\hatW{{\widehat W}}
\def\tildeW{{\widetilde W}}

\def\C{\mathbb{C}}

\def\Z{\mathbb{Z}}
\def\N{\mathbb{N}}

\def\qchoose#1#2{\left[\begin{array}{c} #1\\ #2\end{array}\right]}
\def\ideal{\unlhd}

\begin{document}
\title{Whittaker modules for $U_q(\mathfrak{sl}_3)$}
\author{Xiangqian Guo, Xuewen Liu and Limeng Xia}
\address{Xiangqian Guo, School of Mathematics and Statistics, Zhengzhou University, Zhengzhou 450001, Henan, P. R. China}
\email{guoxq@zzu.edu.cn}
\address{Xuewen Liu, School of Mathematics and Statistics, Zhengzhou University, Zhengzhou 450001, Henan, P. R. China}
\email{liuxw@zzu.edu.cn}
\address{Limeng Xia, Institute of Applied System Analysis, Jiangsu University, Zhenjiang  212013, Jiangsu, P. R. China}
\email{xialimeng@ujs.edu.cn}
\date{}

\maketitle

\begin{abstract}
In this paper, we study the Whittaker modules for the quantum enveloping algebra $U_q(\sl_3)$ with respect to a fixed Whittaker function. 
We construct the universal Whittaker module, find all its Whittaker vectors and investigate the submodules generated by subsets of Whittaker vectors and corresponding quotient modules. We also find Whittaker vectors and determine the irreducibility of these quotient modules and show that they exhaust all irreducible Whittaker modules. Finally, we can determine all maximal submodules of the universal Whittaker module.
The Whittaker model of $U_q(\sl_3)$ are quite different from that of $U_q(\sl_2)$ and finite-dimensional simple Lie algebras, since the center of our algebra is not a polynomial algebra. 
\end{abstract}



\section{Introduction}

Whittaker modules were first introduced and studied by B. Kostant [Ko] for finite-dimensional
complex semisimple Lie algebras. These modules are related to the Whittaker equations 
(see \cite{WW}, P.337) that arise in the study of the corresponding representations of the associated Lie groups. 
Whittaker modules for finite-dimensional simple complex Lie algebras are also related
to parabolic induction and BGG category $\mathcal{O}$ (see for example, \cite{Ba, KM, MD1, MD2, MS1, MS2}).
%
In Block's classification of all simple $\sl_2$-modules in [Bl],
all simple modules fall into two classes: weight modules (including highest and lowest weight modules) and modules on which the action of the Cartan subalgebra is torsion-free (including Whittaker modules). 
%

Recently, many authors studied the Whittaker modules for other algebras, for example, 
quantum groups (\cite{On1, On2, Se}), generalized Weyl algebras (\cite{BO}), affine Lie algebras (\cite{Ch, ALZ}), Virasoro algebra (\cite{OW, LGZ}), and some other algebras (see for example, \cite{OW2, BM, GL} and references therein.)

In particular, Whittaker modules for the quantum group $U_q(\sl_2)$ were fully investigated in \cite{On1}, where Whittaker 
modules with respect to a fixed Whittaker function are in one-to-one correspondence with the ideals of the center of $U_q(\sl_2)$. This is an analogue of the classical results in the Lie algebra setting due to Kostant \cite{Ko}. 

In present paper, we will consider the Whittaker model for $U_q(\sl_3)$. 
The main difference is that there does not exist non-singular Whittaker function for $U_q(\sl_3)$ duo to the quantum Serre relations ((see Section 3)). So we consider the Whittaker modules with respect to a singular Whittaker function. In this case, the Whittaker module theory of $U_q(\sl_3)$ turns out to be much different, since the algebra itself is more complicated as well as the center of $U_q(\sl_3)$ is no longer a polynomial algebra, which cause much trouble for our discussion.

The paper is organized as follows. In Section 2, we present some basic notation and terminology for the 
quantum group $U_q(\sl_3)$. Then we start with the construction of universal Whittaker module $M(\eta)$
for some Whittaker function $\eta$ and make some preparations for later use in Section 3. 

In Section 4, we define certain submodules generated by subsets of Whittaker vectors in $M(\eta)$ and prove some technique results for them, which lead to two consequences. The first is that we can determine the Whittaker vectors in $M(\eta)$ and its quotient modules corresponding to the previous submodules (Theorem \ref{vector}). 
The second one allows us to divide submodules and quotient modules of $M(\eta)$ into two classes: 
critical and non-critical ones (see definitions above Theorem \ref{vector}). 
In particular, we construct quotient modules $V(\eta; \k,c)$ with two parameters $\k,c\in\C, \k\neq0$, 
which we show is irreducible if and only if $(\k,c)$ is non-critical, see Theorem \ref{non-critical}.

The last section is devoted to the critical case, where we completely determine the submodule structure of $V(\eta;\k,c)$, which admits a unique composite series of length at most $2$ (Theorem \ref{critical}). 
In particular, the homomorphisms between two modules $V(\eta;\k,c)$ are computed (Corollary \ref{EndV}) and 
the irreducible quotient modules are determined.  
Finally, we can give an explicit description of all maximal submodules of the universal Whittaker module $M(\eta)$ and obtain a characterization of irreducible Whittaker modules in terms of quotient modules of $M(\eta)$.


\section{The algebra $U_q(\mathfrak{sl}_3)$}

Throughout this paper, we denote by $\Z$, $\N$, $\Z_+$ and $\C$ the
sets of integers, positive integers, nonnegative integers and
complex numbers respectively. All vector spaces and algebras are
over $\C$, although most of the arguments work for any fields with characteristic $0$.
We always fix a nonzero $q\in\C$ which is not a root of unity.

First recall the Jimbo presentation of $U_q(\mathfrak{sl}_3)$ (see \cite{Jan}). The
algebra is an associative algebra generated by $E_1, E_2, F_1, F_2, K_1, K_2, K_1^{-1}, K_2^{-1}$ 
subject to the following relations. 
For $i, j \in \{1, 2\}$, the commutator relations:
\begin{align}\label{comm}
& [E_i, F_j] = \delta_{ij} \frac{K_i - K_i^{-1}}{q-q^{-1}}, \\
& K_i^{\pm 1} E_j = q^{\pm a_{ij}} E_j K_i^{\pm 1}, \\
& K_i^{\pm 1} F_j = q^{\mp a_{ij}} F_j K_i^{\pm 1};
\end{align} and the quantum Serre relations:
\begin{align}\label{serre}
& E_i^2 E_j - (q+q^{-1}) E_i E_j E_i + E_j E_i^2, \quad |i-j|=1, \\
& F_i^2 F_j - (q+q^{-1}) F_i F_j F_i + F_j F_i^2, \quad |i-j|=1,
\end{align}
where $a_{11}=a_{22}=2$ and $a_{12}=a_{21}=-1$.


The algebra $U=U_q(\sl_3)$ admits a triangular decomposition
$$U=U^-\otimes U^0\otimes U^+,$$
where $U^{+}$ is the subalgebra generated by $E_1, E_2$, $U^{-}$ is
the subalgebra generated by $F_1, F_2$ and $U^{0}=\C[K_1^{\pm1}, K_2^{\pm 1}]$.
For convenience, we denote $E_{3}=E_1E_2-q^{-1}E_2E_1$
and $F_{3}=F_1F_2-qF_2F_1$. Then $U^+$ admits a PBW basis
$$\{E_{3}^{r_{3}}E_2^{r_2}E_1^{r_1}, r_1, r_2, r_{12}\in\Z_+\},$$
and $U^-$ admits a PBW basis
$$\{F_{3}^{r_{3}}F_2^{r_2}F_1^{r_1}, r_1, r_2, r_{12}\in\Z_+\}.$$

For $E_3, F_3$, we have the following easily-verified commutator relations:
\begin{align}\label{E_3 F_3}
 & K_iE_3=qE_3K_i, \ K_iF_3=q^{-1}F_3K_i; \quad i=1,2 \\
 & F_1F_3=q^{-1}F_3F_1, \ F_2F_3=qF_3F_2, \ [E_1, F_3]=F_2K_1^{-1}, \ [E_2, F_3]=-K_2F_1;\\
 & E_1E_3=qE_3E_1, \ E_2E_3=q^{-1}E_3E_2, \ [F_1, E_3]=E_2K_1^{-1}, \ [F_2, E_3]=-K_2E_1.
\end{align}

\section{Whittaker modules for $U_q(\mathfrak{sl}_3)$}

For convenience we always denote $U=U_q(\sl_3)$ throughout the paper.
%
%

Fix any nonzero algebra homomorphism $\eta: U^+\ra \C$, which we call a
\textbf{Whittaker function}. 
 In particular, $\eta(E_{3})=(1-q^{-1})\eta(E_1)\eta(E_2)$
and hence $\eta$ is determined by its values on $E_1$ and $E_2$. The
Whittaker function $\eta$ is called \textbf{non-singular} if
$\eta(E_1)$ and $\eta(E_2)$ are both nonzero and is called
\textbf{singular} otherwise. From the quantum Serre relations
\eqref{serre}, we see easily that $\eta(E_1)\eta(E_2)=0$.
Consequently, there do not exist non-singular Whittaker functions
and we will consider the singular case. \textbf{In the rest of this
paper, we will always suppose that $\eta(E_1)=\a\neq0$ and
$\eta(E_2)=0$} without loss of generality due to the symmetry
between $E_1$ and $E_2$.

Now take any $U$-module $W$. A vector $w\in W$ is called a
\textbf{Whittaker vector of type $\eta$} if $xw=\eta(x)v$ for all
$x\in U^+$ and $W$ is called a \textbf{Whittaker module of type
$\eta$ or $(\a,0)$} if it is generated by a nonzero Whittaker vector of the same
type. In what follows, all Whittaker modules and Whittaker vectors
are assumed of type $\eta$ if not stated otherwise.

We first define the universal Whittaker module. Define a
$1$-dimensional $U^+$-module $\C v_\eta$ given by $xv_\eta=\eta(x)v_\eta$ for all
elements $x\in U^+$. Then we form the induced $U$-module
$$M(\eta)=U\otimes_{U^+}\C v_\eta,$$
which is called the \textbf{universal Whittaker module of type
$\eta$}. Obviously, any Whittaker module is isomorphic
to a homomorphic image of the universal Whittaker module.



Let $U^{(1)}$ be the subalgebra of $U$ generated by $E_1,K_1,F_1$.
Denote $K=K_1K_2^2$. It is clear that $K$ commutes with $E_1, F_1$
and $K_1$. Recall the Casimir element of $U^{(1)}$ 
$$C_1=F_1E_1+\frac{qK_1+q^{-1}K_1^{-1}}{(q-q^{-1})^2}.$$
Since $[C_1, K]=0$, we have the polynomial algebra $\C[K^{\pm1},
C_1]$ and $[U^{(1)}, \C[K^{\pm1}, C_1]]=0$.

Note that $E_1w=\a w$ and $E_2w=0$ for any $w\in\C[K^{\pm1}, C_1]v_\eta$, 
so $w$ is a Whittaker vector. 
In next section, we will show that $\C[K^{\pm1}, C_1]v_\eta$ is just the set of all Whittaker vectors
(see Theorem \ref{vector}).

By the PBW Theorem, $M(\eta)$ has a basis
$$\{F_{2}^{j}F_3^{k}K_2^{l}K_1^pF_1^{r}v_\eta\ |\ l, p\in\Z, j,k, r\in\Z_+\}.$$
Noticing $E_1v_\eta=\a v_\eta$ and the definition of 
$C_1$ and $K$, the module $M(\eta)$ has another basis
$$\{F_{2}^{j}F_3^{k}K_2^{l}K^{p}C_1^{r}v_\eta\ |\ l, p\in\Z, j,k, r\in\Z_+\}$$

Denote $\Gamma=\Z_+^2\oplus\Z$ and define a total order on $\Gamma$
as follows: for any $\gamma=(j,k,l)$, $\gamma'=(j',k',l')\in
\Gamma$, define $\gamma\geq \gamma'$ if and only if $j+k>j'+k'$, or $j+k=j'+k'$
and $k>k'$, or $j+k=j'+k', k=k'$ and $l\succ l'$, where $l\succeq l'$ means $|l|>|l'|$ or $l=-l'\geq0$.
We have the induced filtration on $M=M(\eta)$ given by:
$$M_{j,k,l}=\sum_{(j',k',l')\leq(j,k,l)}F_{2}^{j'}F_3^{k'}K_2^{l'}\C[K^{\pm1}, C_1]v_\eta$$ and
$$M'_{j,k,l}=\sum_{(j',k',l')<(j,k,l)}F_{2}^{j'}F_3^{k'}K_2^{l'}\C[K^{\pm1}, C_1]v_\eta.$$
For any element $w\in M_{j,k,l}\setminus M'_{j,k,l}$ we denote $\deg(w)=(j,k,l)$, called the {\bf degree} of $w$.
We also denote 
$$M_n=\sum_{j+k=n}F_{2}^{j}F_3^{k}K_2^{l}\C[K^{\pm1}, C_1]v_\eta,\ \text{and}\ M'_{n}=\sum_{n'<n}M_{n'},\ \forall\ n\in\Z_+.$$

As usual, for any $k\in\Z$, we denote 
$$[k]=\frac{q^{k}-q^{-k}}{q-q^{-1}} \quad\text{and}\quad [K_i;k]=\frac{q^{k}K_i-q^{-k}K_i^{-1}}{q-q^{-1}},\ i=1,2.$$
The following identities and relations will be used from time t time.

\begin{lemma}\label{compute} For any $j,k,n\in\Z_+$ and $l\in\Z$, we have
\begin{enumerate}
  \item $F_1F_2^j=[j]F_2^{j-1}F_3+q^jF_2^jF_1;$
  \item $[E_1, F_2^jF_3^k]=[k]F_2^{j+1}F_3^{k-1}K_1^{-1}$;
  \item $[E_2, F_2^jF_3^k]=[j]F_2^{j-1}F_3^k[K_2;1-j-k]-q^{1-k}[k]F_2^jF_3^{k-1}K_2F_1$.
 \item $(E_1-q^l\a)M_{j,k,l}\subseteq M'_{j,k,l}$ and $E_2M_{j,k,l}\subseteq M'_{j,k,l}$, where  $M'_{0,0,0}=0$.
 \item $E_1M_n\subseteq M_n$ and $E_2M_n\subseteq M_{n-1}$, where $M_{-1}=0$.
\end{enumerate}
\end{lemma}

\begin{proof}
All assertions can be checked straightforward.
\end{proof}


For any ideal $I$ of the commutative algebra $\C[K^{\pm1}, C_1]$, we denote
by $W(\eta,I)$ the $U$-submodule of $M(\eta)$ generated by $Iv_\eta$, that is, $W(I,\eta)=UIv_\eta$. Set $V(\eta,I)=M(\eta)/W(\eta,I)$. We will see later that $W(\eta,I), I\ideal\C[K^{\pm1},C_1]$, are precisely the submodules
generated by subsets of Whittaker vectors of type $(\a,0)$ in $M(\eta)$ (see Remark \ref{vector gen mod}).

The following obvious results will be used frequently in the rest of the paper.

\begin{lemma} \label{W(eta,I)} 
For any ideal $I\ideal \C[K^{\pm1}, C_1]$, we have
$$W(\eta,I)=\bigoplus_{k,j\in\Z_+, l\in\Z}F_2^{j}F_3^{k}K_2^lIv_\eta.$$
\end{lemma}

\begin{lemma}\label{claim}
Let $I$ be an ideal of $\C[K^{\pm1},C_1]$ and $w=\sum_{(j,k,l)\in\Gamma}F_2^{j}F_3^{k}K_2^{l}Q_{j,k,l}v_\eta$, where $Q_{j,k,l}\in\C[K^{\pm1},C_1]$. Then $w\in W(\eta, I)$ if and only if $Q_{j,k,l}\in I$ for any $(j,k,l)\in \Gamma$.
\end{lemma}

\begin{lemma} \label{claim'} 
Suppose that $N$ is a subspace of $M(\eta)$ stable under the action of $E_1$, then $\sum_{i=1}^kK_2^iQ_iv_\eta\in N$ for some $Q_i\in\C[K^{\pm1}, C_1]$ implies $Q_iv_\eta\in N$ for all $i=1,\dots,k$.
\end{lemma}

\begin{proof}
Suppose $\sum_{i=1}^kK_2^iQ_iv_\eta\in N$, then applying $E_1^j$ on it, we get
$$E_1^j\sum_{i=1}^kK_2^iQ_iv_\eta=\sum_{i=1}^kq^{ij}K_2^iE_1^jQ_iv_\eta=\a^j \sum_{i=1}^kq^{ij}K_2^iQ_iv_\eta\in N,\ \forall\ j\in\N.$$
Using the property of the Vandermonde's determinant, we deduce
$K_2^iQ_iv_\eta\in N$.
\end{proof}


\section{Whittaker vectors and non-critical irreducible Whittaker modules}

Let notation be as before. Since irreducible Whittaker modules are just irreducible quotient modules
of the universal Whittaker module $M(\eta)$, we need to determine maximal submodules of $M(\eta)$.
We start with the study of structure of submodules $W(\eta,I)$, which is generated by $Iv_\eta$ for an ideal $I\ideal\C[K^{\pm1}, C_1]$.

\begin{lemma}\label{E_1 action} 
Let $I$ be an ideal of $\C[K^{\pm1}, C_1]$ and $n_0\in\Z_+$. Suppose
\begin{equation}\label{u_n_0,1}
u=\sum_{k=0}^{n_0}\sum_{l\in\Z}F_2^{n_0-k}F_3^{k}K_2^{l}Q_{k,l}v_\eta\in M_{n_0},
\end{equation}
where only finitely many $Q_{k,l}\in\C[K^{\pm1}, C_1]$ are nonzero.
Suppose $u\notin W(\eta,I)$ in addition, then $(E_1-\a q^{l_0})u\in W(\eta,I)$ for some $l_0\in\Z$ if and only if 
\begin{equation}\label{u_n_0-2}
u\equiv\sum_{k=0}^{k_0}\sum_{j=0}^{k_0-k} q^{\frac{1}{2}k(k-1)}\qchoose{j+k}{k}F_2^{n_0-k}F_3^kK_2^{l_0+2j}\big(\a(1-q^2)q^{l_0+j}K\big)^kQ_{l_0+2j+2k}v_\eta\hskip-5pt\mod W(\eta,I),
\end{equation}
for some $0\leq k_0\leq n_0$, where $Q_{l_0+2i}=Q_{0,l_0+2i}$ for $0\leq i\leq k_0$ with $Q_{l_0+2k_0}\notin I$, the $q$-binomial coefficients $\qchoose{j+k}{k}=\frac{[j+k]!}{[j]![k]!}$, $[k]!=[k][k-1]\cdots[1]$ and $[0]!=1$ for $k,j\in\Z_+$.
\end{lemma}

\begin{proof} 
If $u$ is of the form in \eqref{u_n_0-2}, we can check straightforward that $(E-\a q^{l_0})u\in W(\eta, I)$. 
Now suppose on the contrary $(E-\a q^{l_0})\in W(\eta, I)$ for some $l_0\in\Z$.

By Lemma \ref{compute}, we can compute fro any $\ (j,k,l)\in\Gamma$ and $Q\in\C[K^{\pm1},C_1]$ that
\begin{equation}\label{E_1u}\begin{split}
  E_1F_2^{j}F_3^{k}K_2^{l}Qv_\eta= &[E_1,F_2^{j}F_3^{k}]K_2^{l}Qv_\eta+F_2^{j}F_3^{k}E_1K_2^{l}Qv_\eta\\
    = &[k]F_2^{j+1}F_3^{k-1}K_2^{l+2}K^{-1}Qv_\eta +\a q^{l}F_2^{j}F_3^{k}K_2^{l}Qv_\eta.\\
\end{split}\end{equation}

Let $k_0$ be the maximal $k\in\Z_+$ such that $Q_{k,l}\notin I$ for some $l\in\Z$.
Denote $S=\{l\in\Z\ |\ Q_{k_0,l}\notin I\}$ which is finite. Without loss of generality, we assume $Q_{k_0,l}=0$ for $l\notin S$. There exists $s\in\N$ such that
$$S=\{l_1,l_2,\dots,l_s\}\ \text{with}\ l_1\succ l_2\succ\dots\succ l_s.$$ 
Then we have
\begin{equation}\label{E_1-alpha}\begin{split}
&(E_1-\a q^{l_0})u\\
   \equiv &\sum_{l\in S}[k_0]F_2^{n_0-k_0+1}F_3^{k_0-1}K_2^{l+2}K^{-1}Q_{k_0,l}v_\eta+\a\sum_{l\in S}(q^{l}-q^{l_0})F_2^{n_0-k_0}F_3^{k_0}K_2^{l}Q_{k_0,l}v_\eta\\
   \equiv & \a\sum_{i=1}^s(q^{l_i}-q^{l_0})F_2^{j_0}F_3^{k_0}K_2^{l_i}Q_{k_0,l_i}v_\eta\mod M'_{n_0-k_0,k_0,0}.
\end{split}\end{equation}
Then $(E_1-\a q^{l_0})u\in W(\eta,I)$ implies $s=1$ and $l_1=l_0$.

Now we have $Q_{k_0,l}=0$ for $l\neq l_0$.
Applying $E_1$ to $u$, by \eqref{E_1u} we deduce
\begin{equation*}\begin{split}
      & E_1u-\a q^{l_0}u\\
    = & \sum_{k=0}^{k_0}\sum_{l\in\Z}[k]F_2^{n_0-k+1}F_3^{k-1}K_2^{l+2}K^{-1}Q_{k,l}v_\eta
               +\a\sum_{k=0}^{k_0-1}\sum_{l\in\Z}(q^l-q^{l_0})F_2^{n_0-k}F_3^{k}K_2^lQ_{k,l}v_\eta\\
    = & \sum_{k=1}^{k_0}\sum_{l\in\Z}F_2^{n_0-k+1}F_3^{k-1}\big([k]K_2^{l+2}K^{-1}Q_{k,l}+\a(q^l-q^{l_0})K_2^lQ_{k-1,l}\big)v_\eta.
\end{split}\end{equation*}
By hypothesis on $u$, we see that
\begin{equation}\label{Q_k,l condition}
\sum_{l\in\Z}[k]K_2^{l+2}K^{-1}Q_{k,l}-\a\sum_{l\in\Z}(q^{l_0}-q^{l})K_2^lQ_{k-1,l}\in \C[K_2]I,\ \quad \forall\ k=1,2,\dots,k_0.
\end{equation}
Noticing that $Q_{k_0,l}=\delta_{l,l_0}Q_{k_0,l_0}$, the above formula for $k=k_0$ implies $Q_{k_0-1,l}=0$ for all $l\neq l_0, l_0+2$.
Inductively, we can deduce that $Q_{k_0-i,l}=0$ for all $l\neq l_0, l_0+2,\dots, l_0+2i$.
Then the equation \eqref{Q_k,l condition} becomes
\begin{equation*}
\sum_{j=0}^i[k_0-i]K_2^{l_0+2j+2}K^{-1}Q_{k_0-i,l_0+2j}-\a\sum_{j=0}^{i+1}q^{l_0}(1-q^{2j})K_2^{l_0+2j}Q_{k_0-i-1,l_0+2j}\in \C[K_2]I,
\end{equation*}
or, more precisely,
\begin{equation*}
Q_{k_0-i,l_0+2j}-\frac{\a q^{l_0}(1-q^{2j+2})}{[k_0-i]} KQ_{k_0-i-1,l_0+2j+2}\in I,
\end{equation*}
for all $i=0,1,\dots,k_0-1$ and $j=0,1,\dots,i$.
Replacing each $Q_{k-i,l_0+2j}$ with $Q_{k-i,l_0+2j}$ plus a suitable elements in $I$, we may assume that
\begin{equation*}
Q_{k_0-i,l_0+2j}=\frac{\a q^{l_0+j}(1-q^2)[j+1]}{[k_0-i]} KQ_{k_0-i-1,l_0+2j+2},
\end{equation*}
for all $j=0,1,\dots,i$ and $i=0,1,\dots,k_0-1$.

Now setting $Q_{l_0+2i}=Q_{0,l_0+2i}$, we can deduce that
\begin{equation}\label{Q_k,l}
Q_{k,l_0+2j}=q^{\frac{1}{2}k(2j+k-1)}\qchoose{j+k}{k}\big(\a(1-q^2)q^{l_0}K\big)^kQ_{l_0+2j+2k},
\end{equation}
for $0\leq j\leq k_0-k$, $0\leq k\leq k_0$. Substituting this formula into \eqref{u_n_0,1} proves the lemma.
\end{proof}

For any $n\in\Z_+, l\in\Z$ and $Q\in\C[K^{\pm1}, C_1]$, we set 
\begin{equation}\label{u(n,l,Q)}
u(n,l,Q)=q^{2nl}K_2^l\sum_{k=0}^{n}\sum_{j=0}^{n-k}a_{kj}(n)F_2^{n-k}F_3^k
K_2^{2j}K^{-2j-k}Qv_\eta,
\end{equation}
where $a_{kj}(n)=(-1)^j\a^k(q^2-1)^kq^{j(n-3)}q^{\frac{1}{2}k(2n+2j+k-7)}\qchoose{n}{k}\qchoose{n-k}{j}$.
It is easy to check that $u(n,l,Q)=q^{2nl}K_2^lu(n,0,Q)$ for any $n\in\Z_+, l\in\Z$ and $Q\in\C[K^{\pm1},C_1]$.

\begin{remark}\label{u(n,0,Q)}
Later, we will see $u(n,0,Q)=g^nQv_\eta$ for some element $g\in\C[F_2,F_3,K_2,K]$ 
(see Corollary \ref{u(n)}). 
This much simpler expression is motivated by arguments in Section 5 (see Remark \ref{g_n}) and 
is hard to find out right now. So we leave it to the next section and use the expression in \eqref{u(n,l,Q)}, which is more convenient for our present purpose.
\end{remark}

\begin{lemma}\label{E_2 action} 
Let $I$ be an ideal of $\C[K^{\pm1}, C_1]$ and $n_0\in\Z_+$. Let $u\in M_{n_0}$ be as in \eqref{u_n_0,1}.
\begin{itemize} 
\item[(1)] If $n_0=0$, then $(E_1-\a q^{l_0})u, E_2u\in W(\eta,I)$ for some $l_0\in\Z$ if and only if $u=K_2^{l_0}Qv_\eta$ for some $Q\in\C[K^{\pm1}, C_1]$.
\item[(2)] If $n_0\geq 1$, then $(E_1-\a q^{l_0})u, E_2u\in W(\eta,I)$ for some $l_0\in\Z$ if and only if 
$u=u(n_0,l_0,Q_{l_0})$ and
$$E_2u\equiv q^{2l_0(n_0-1)}K_2^{l_0}\sum_{k=0}^{n_0-1}\sum_{j=0}^{n_0-k-1}b_{kj}(n_0)F_2^{n_0-k-1}F_3^kK_2^{2j+1}K^{-k-2j-1}h_{n_0}(K,C_1)Q_{l_0}v_\eta$$
$\mod W(\eta,I)$, where $b_{kj}(n_0)=q^{n_0-j-1}(q^2-1)^{-1}[n_0]a_{k,j}(n_0-1)$, $Q_{l_0}=Q_{0,l_0}$ and 
$$h_{n_0}(K,C_1)=q^{3-2n_0}K+q^{2n_0-3}K^{-1}-(q-q^{-1})^2C_1.$$ 
\end{itemize} 
\end{lemma}

\begin{proof} 
The sufficiency can be verified straightforward, though it is quite tedious. We then prove the necessity. 

By Lemma \ref{E_1 action}, we may write $u$ in the form of \eqref{u_n_0-2}. More precisely, we assume that $u=\sum_{k=0}^{k_0}F_2^{n_0-k}F_3^kP_kv_\eta$ 
for some $k_0\leq n_0$, where the polynomial $P_k\in\C[K_2^{\pm1},K^{\pm1}, C_1]$ is given by
$P_k=\sum_{j=0}^{k_0-k}K_2^{l_0+2j}Q_{k,l_0+2j}$, where
\begin{equation}\label{Q_k}
\quad Q_{k,l_0+2j}=q^{\frac{1}{2}k(k-1)}\qchoose{j+k}{k}\big(\a(1-q^2)q^{l_0+j}K\big)^kQ_{0,l_0+2j+2k}.
\end{equation}
As before, we denote $Q_{l_0+2j}=Q_{0,l_0+2j}$ for short and hence $Q_{l_0+2k_0}\notin I$.

\noindent(1) If $n_0=0$, we have $k_0=0$ and $u=K_2^{l_0}Q_{l_0}v_\eta$. Hence $E_2u=0$ always holds.

\noindent(2) Assume $n_0\geq1$ from now on. 
By Lemma \ref{compute}, we compute the application of $E_2$ on $u$
\begin{equation}\label{E_2u}\begin{split}
E_2u
  = & \sum_{k=0}^{k_0}[E_2, F_2^{n_0-k}F_3^k]P_kv_\eta+\sum_{k=0}^{k_0}F_2^{n_0-k}F_3^kE_2P_kv_\eta\\
  = & \sum_{k=0}^{k_0}[n_0-k]F_2^{n_0-k-1}F_3^k[K_2;1-n_0]P_kv_\eta
     -\sum_{k=1}^{k_0}q^{1-k}[k]F_2^{n_0-k}F_3^{k-1}K_2F_1P_kv_\eta\\
  = & \sum_{k=0}^{k_0}F_2^{n_0-k-1}F_3^k\big([n_0-k][K_2;1-n_0]P_k
  -q^{-k}[k+1]K_2F_1P_{k+1}\big)v_\eta,
\end{split}\end{equation}
where $P_{k_0+1}=0$.
By the assumption $E_2u\in W(\eta,I)$ and Lemma \ref{claim}, we see,
\begin{equation}\label{P_k condition}
[n_0-k][K_2;1-n_0]P_k-q^{-k}[k+1]K_2F_1P_{k+1}v_\eta\in W(\eta,I),\ 
\forall\ 0\leq k\leq k_0.
\end{equation}
If $k_0<n_0$, then take $k=k_0$ in the above formula, and using the fact $P_{k_0+1}=0$, we get
$$[n_0-k_0][K_2;1-n_0]P_{k_0}v_\eta\in W(\eta,I),$$
which implies $Q_{l_0+2k_0}\in I$ by Lemma \ref{claim'}, contradicting the fact $Q_{l_0+2k_0}\notin I$.

So we must have $k_0=n_0$ and hence the equation \eqref{P_k condition} can be rewritten as:
\begin{equation*}\aligned
&[n_0-k][K_2;1-n_0]\sum_{j=0}^{n_0-k}K_2^{l_0+2j}Q_{k,l_0+2j}v_\eta-q^{-k}[k+1]K_2F_1\sum_{j=0}^{n_0-k-1}K_2^{l_0+2j}Q_{k+1,l_0+2j}v_\eta\\
=&[n_0-k][K_2;1-n_0]\sum_{j=0}^{n_0-k}K_2^{l_0+2j}Q_{k,l_0+2j}v_\eta\\
&\hskip1cm-q^{-k}[k+1]\sum_{j=0}^{n_0-k-1}\a^{-1}q^{-l_0-2j}K_2^{l_0+2j+1}F_1E_1Q_{k+1,l_0+2j}v_\eta\\
=&[n_0-k][K_2;1-n_0]\sum_{j=0}^{n_0-k}K_2^{l_0+2j}Q_{k,l_0+2j}v_\eta\\
&\hskip1cm-\a^{-1}q^{-k-l_0}[k+1]\sum_{j=0}^{n_0-k-1}q^{-2j}K_2^{l_0+2j+1}\Big(C_1-\frac{qK_1+q^{-1}K_1^{-1}}{q-q^{-1}}\Big)Q_{k+1,l_0+2j}v_\eta\in W(\eta,I)\\
\endaligned\end{equation*}
%
for all $0\leq k\leq n_0-1$.  Noticing the fact $K=K_1K_2^2$, we have
\begin{equation*}
\begin{split}
& \a q^{l_0+k}(q^2-1)[n_0-k]\left(\sum_{j=0}^{n_0-k}q^{-n_0+2}K_2^{2j+1}Q_{k,l_0+2j}-\sum_{j=-1}^{n_0-k-1}q^{n_0}K_2^{2j+1}Q_{k,l_0+2j+2}\right)\\
& \hskip0.3cm +[k+1]\left(\sum_{j=-1}^{n_0-k-2}q^{1-2j}K_2^{2j+1}KQ_{k+1,l_0+2j+2}+\sum_{j=1}^{n_0-k}q^{3-2j}K_2^{2j+1}K^{-1}Q_{k+1,l_0+2j-2}\right.\\
& \hskip2.5cm \left.-\sum_{j=0}^{n_0-k-1}(q^2-1)^2q^{-2j}K_2^{2j+1}C_1Q_{k+1,l_0+2j}\right)\in \C[K_2^{\pm1}]I
\end{split}\end{equation*}
for $0\leq k\leq n_0-1$. 

Considering the coefficients of $K_2^{2j+1}$, we obtain that
\begin{equation*}\begin{split}
& \a q^{l_0+k}(q^2-1)[n_0-k]\Big(q^{-n_0+2}Q_{k,l_0+2j}-q^{n_0}Q_{k,l_0+2j+2}\Big)\\
& \hskip1cm +[k+1]\Big(q^{1-2j}KQ_{k+1,l_0+2j+2}+q^{3-2j}K^{-1}Q_{k+1,l_0+2j-2}\\
& \hskip6cm -(q^2-1)^2q^{-2j}C_1Q_{k+1,l_0+2j}\Big)\in I,
\end{split}\end{equation*}
for $0\leq k\leq n_0-1$ and $-1\leq j\leq n_0-k$, where we have made the convention $Q_{k,l_0+2j}=0$ if $j<0$ or $k+j>n_0$. 
Substitute \eqref{Q_k,l} into the above equation, we get
\begin{equation}\label{coeff K_2^j}\begin{split}
&q^{j+k}[n_0-k]\Big(q^{-n_0+2}\qchoose{j+k}{k}Q_{l_0+2j+2k}-q^{n_0+k}\qchoose{j+k+1}{k}Q_{l_0+2j+2k+2}\Big)\\
&\hskip0.3cm -[k+1]\Big(q^{2k+2}\qchoose{j+k+2}{k+1}K^2Q_{l_0+2j+2k+4}+q^{2}\qchoose{j+k}{k+1}Q_{l_0+2j+2k}\\
& \hskip5cm -(q^2-1)^2q^{k}\qchoose{j+k+1}{k+1}C_1KQ_{l_0+2j+2k+2}\Big)\in I,
\end{split}\end{equation}
for all $k=0,\dots,n_0-1$, where $\qchoose{j}{k}=0$ if $j<k$ and $Q_{l_0+2i}=0$ if $i<0$ or $i>n_0$.

Taking $j=-1$ in \eqref{coeff K_2^j}, we get
\begin{equation*}\begin{split}
& q^{n_0-3}[n_0-k]Q_{l_0+2k}+[k+1]K^2Q_{l_0+2k+2}\in I,\ \forall\ 0\leq k\leq n_0-1.
\end{split}\end{equation*}
Replacing $Q_{l_0+2k}$  with some element in $Q_{l_0+2k}+I$, we may assume that
\begin{equation*}\begin{split}
& q^{n_0-3}[n_0-k]Q_{l_0+2k}+[k+1]K^2Q_{l_0+2k+2}=0,\ \forall\ 0\leq k\leq n_0-1.
\end{split}\end{equation*}
By induction, we can obtain that
\begin{equation}\label{Q_k reduction}\begin{split}
& Q_{l_0+2k}=
\qchoose{n_0}{k}
(-q^{n_0-3}K^{-2})^kQ_{l_0},\ \forall\ 0\leq k\leq n_0.
\end{split}\end{equation}
Substitute the above equation into \eqref{u_n_0-2}, and simplifying it, we get $u=u(n_0,l_0,Q_{l_0})$.

%
%
%
%

At last, substitute \eqref{Q_k reduction} into \eqref{coeff K_2^j}, we obtain 
\begin{equation*}\begin{split}
&q^{j+k}[n_0-k]\Big(q^{-n_0+2}\qchoose{j+k}{k}\qchoose{n_0}{j+k}+q^{2n_0+k-3}\qchoose{j+k+1}{k}\qchoose{n_0}{j+k+1}K^{-2}\Big)Q_{l_0}\\
&\hskip0.3cm -[k+1]\Big(q^{2n_0+2k-4}\qchoose{j+k+2}{k+1}\qchoose{n_0}{j+k+2}K^{-2}+q^{2}\qchoose{j+k}{k+1}\qchoose{n_0}{j+k}\\
& \hskip2.5cm +(q^2-1)^2q^{n_0+k-3}\qchoose{j+k+1}{k+1}\qchoose{n_0}{j+k+1}C_1K^{-1}\Big)Q_{l_0}\in I,
\end{split}\end{equation*}
for all $k=0,\dots,n_0-1$ and $j=-1,0\dots,n_0-k$. 
We can simplify the above formula as
\begin{equation}\label{Q_l condition}\begin{split}
&q^{j+k}\Big(q^{-n_0+2}\qchoose{n_0-k}{j}+q^{2n_0+k-3}\qchoose{n_0-k}{j+1}K^{-2}\Big)Q_{l_0}\\
&\hskip0.3cm -\Big(q^{2n_0+2k-4}\qchoose{n_0-k-1}{j+1}K^{-2}+q^{2}\qchoose{n_0-k-1}{j-1}\\
& \hskip2.5cm +(q^2-1)^2q^{n_0+k-3}\qchoose{n_0-k-1}{j}C_1K^{-1}\Big)Q_{l_0}\in I,
\end{split}\end{equation}
for all $k=0,\dots,n_0-1$ and $j=-1,0\dots,n_0-k$. This further gives
\begin{equation}\label{Q_l condition}\begin{split}
&[j+1]q^2\Big([n_0-k]q^{j+k-n_0}-[j]\Big)Q_{l_0}\\
&\hskip1cm +[n_0-k-j]q^{2n_0+2k-4}\Big([n_0-k]q^{j+1}-[n_0-k-j-1]\Big)K^{-2}Q_{l_0}\\
& \hskip2cm -(q^2-1)^2q^{n_0+k-3}[n_0-k-j][j+1]C_1K^{-1}Q_{l_0}\in I.
\end{split}\end{equation}
Using the identity $[n]q^j-[j]q^n=[n-j]$, we have
\begin{equation}\label{Q_l condition}\begin{split}
&[j+1][n_0-k-j]q^{n_0+k-1}\big(q^{3-2n_0}+q^{2n_0-3}K^{-2}-(q-q^{-1})^2C_1K^{-1}\big)Q_{l_0}\in I,
\end{split}\end{equation}
which is equivalent to 
\begin{equation}\label{Q_l condition}
\big(q^{3-2n_0}K+q^{2n_0-3}K^{-1}-(q-q^{-1})^2C_1\big)Q_{l_0}\in I,
\end{equation}
since $n_0\geq1$. Substitute all these back into \eqref{coeff K_2^j}, we can obtain the
expression of $E_2u$ as in the lemma.
%
%
\end{proof}

For any $n\in\N$, we denote
$$h_n(K,C_1)=q^{3-2n}K+q^{2n-3}K^{-1}-(q-q^{-1})^2C_1,$$
called {\bf critical polynomials}. 
Any pair $(\kappa,c)\in\C^2, \kappa\neq0$, with $h_n(\kappa,c)=0$ for some $n\in\N$, is called a 
{\bf critical pair};
otherwise $(\kappa,c)$ is called {\bf non-critical}.

%

For any ideal $I\unlhd\C[K^{\pm1}, C_1]$ and $n\in\N$, denote 
$$\hat I_n=\{Q\in\C[K^{\pm1},C_1]\ \big|\ h_n(K,C_1)Q\in I\},$$
which is again an ideal of $\C[K^{\pm1}, C_1]$.
An ideal $I$ is called {\bf non-critical} if $I=\hat I_n$ for any $n\in\N$ and is called {\bf critical} otherwise. Denote by $\hat I$ the minimal non-critical ideal containing $I$. 
Clearly $\hat I_n\subseteq\hat I$ and the pair $(\k,c)$ is critical if and only if the maximal ideal of 
$\C[K^{\pm1},C_1]$ generated by $K-\k, C_1-c$ is critical.

\begin{theorem}\label{vector} 
For any $l\in\Z$, the set of all Whittaker vectors of type $(\a q^l, 0)$ in the 
Whittaker module $V(\eta, I)=M(\eta)/W(\eta,I)$ is 
$$\spn\{K_2^l\C[K^{\pm1}, C_1]\bar v_\eta\}\cup \{\bar u(n,l,Q)\ |\ n\in\N, Q\in \hat I_n\setminus I\},$$ 
where $\bar v_\eta, \bar u(n,l,Q)\in V(\eta, I)$ are the images of the elements $v_\eta$ and $u(n,l,Q)$ respectively. In particular, the set of Whittaker vectors of type $(\a q^l, 0)$ in $M(\eta)$ is 
$K_2^l\C[K^{\pm1}, C_1]v_\eta$.
\end{theorem}

\begin{remark}\label{vector gen mod}
In light of Theorem \ref{vector}, we see that $W(\eta,I)$ is just the submodule of $M(\eta)$ generated by  $I v_\eta$, a subset of Whittaker vectors of type $(\a,0)$ in $M(\eta)$. On the other hand, all such submodules are of the form $W(\eta,I)$ for some $I\ideal\C[K^{\pm1},C_1]$.
\end{remark}

For any $U$-submodule $N\subseteq M(\eta)$, set 
$$I(N)=\{Q\in\C[K^{\pm1},C_1]\ |\ Qv_\eta\in N\}.$$
Then $I(N)$ is an ideal of $\C[K^{\pm1}, C_1]$ and $W(\eta,I(N))$ is a submodule of $N$.

\begin{proposition}\label{submodule}
Suppose that $N$ is a submodule of $M(\eta)$ and $I=I(N)$ is non-critical, then $N=W(\eta, I)$.
\end{proposition}


\begin{proof} It is clear $W(\eta, I)\subseteq N$. We now prove $N\subseteq W(\eta, I)$.

Suppose on the contrary that $N\not\subseteq W(\eta, I)$.
Take $w\in N\setminus W(\eta, I)$ with minimal degree.
Set $\deg(w)=(j_0,k_0,l_0)\in\Gamma$ and $n_0=j_0+k_0$.
Suppose
\begin{equation*}
w=\sum_{(j,k,l)\leq(j_0,k_0,l_0)}F_2^{j}F_3^{k}K_2^{l}Q_{j,k,l}v_\eta,
\end{equation*}
with $Q_{j_0,k_0,l_0}\notin I$.
If $n_0=0$, we have $Q_{0,0,l_0}v_\eta\in N$ by Lemma \ref{claim'} and hence $Q_{0,0,l_0}\in I$, contradiction. Thus $n_0\geq1$.

%
Write $w=u+u'$, where $u\in M_{n_0}$ 
and $u'\in M'_{n_0}$. 
Noticing $\deg((E_1-\a q^{l_0})w)<\deg(w)$ and $\deg(E_2w)<\deg(w)$ by Lemma \ref{compute} (4),
we have $(E_1-\a q^{l_0})w, E_2w\in W(\eta, I)$ by the choice of $w$. Then
$$(E_1-\a q^{l_0})w=(E_1-\a q^{l_0})u+(E_1-\a q^{l_0})u'\ \text{and}\ E_2w=E_2u+E_2u',$$
imply $(E_1-\a q^{l_0})u_{n_0}, E_2u_{n_0}\in W(\eta, I)$ by Lemma \ref{compute} (5) and Lemma \ref{claim}.

By Lemma \ref{E_2 action},
we have $u_{n_0}=u(n_0,l_0,Q)$, where $Q=Q_{n_0,0,l_0}$, which corresponds to the polynomial $Q_{l_0}=Q_{0,l_0}$ in Lemma \ref{E_1 action} \eqref{u_n_0,1} if we write $u$ in the form of \eqref{u_n_0,1}, lies in $\hat I_{n_0}=I$. By \eqref{Q_k,l condition} in the proof of Lemma \ref{E_1 action} and 
\eqref{Q_k reduction} in Lemma \ref{E_2 action}, we see that $Q_{j_0,k_0,l_0}$, which corresponds to the polynomial $Q_{k_0,l_0}$
in \eqref{u_n_0,1}, lies in $I$, again contradiction. This completes the proof.
\end{proof}

For any complex number $\kappa, c\in\C, \k\neq0$, we denote by $J(\kappa,c)$ the maximal ideal of $\C[K^{\pm1},C_1]$ generated by $K-\k, C_1-c$. Set moreover $W(\eta;\kappa,c)=W(\eta, J(\kappa,c))$ and $V(\eta;\kappa,c)=M(\eta)/W(\eta; \k,c))$, which has a basis $\{F_2^jF_3^kK_2^l\bar v_\eta\ |\ (j,k,l)\in\Gamma\}$, where $\bar v_\eta$ is the image of $v_\eta$ in $V(\eta; \k,c)$. 
Recall that $J(\k,c)$ is critical if and only if $(\k,c)$ is critical.

\begin{theorem}\label{non-critical}
Suppose $\k, c\in\C$ with $\k\neq0$. Then $W(\eta; \k, c)$ is a maximal submodule if and only if $(\k,c)$ is non-critical; $V(\eta; \k, c)$ is an irreducible Whittaker module of type $\eta$ if and only if $(\kappa,c)$ is non-critical.
%
\end{theorem}

\begin{proof} Denote $J=J(\kappa,c)$.
Suppose that $(\k,c)$ and hence $J$ are non-critical. Let $N$ be a proper submodule containing 
$W(\eta;\kappa,c)=W(\eta;J)$, then $J\subseteq I(N)$, forcing $J=I(N)$. 
By Proposition \ref{submodule}, we have $N=W(\eta;\kappa,c)$.
So $W(\eta;\k,c)$ is maximal and $V(\eta;\k,c)$ is irreducible.

Now suppose that $(\k,c)$ is critical, and take $n\in\N$ such that $h_{n}(\k,c)=0$. 
By Lemma \ref{E_2 action}, we have $(E_1-\a)u(n,0,1)=E_2u(n,0,1)=0$, where $1\in\hat J_{n}\setminus J$ is the constant polynomial in $\C[K^{\pm1},C_1]$. 
Then $Uu(n,0,1)+W(\eta; \k,c)$ is a proper submodule of $M(\eta)$ properly containing $W(\eta;\k,c)$,
which is hence not maximal.
%
%
%
\end{proof}

\section{maximal submodules and critical irreducible Whittaker modules}

We further discuss the submodule structure and determine their irreducible quotients of $V(\eta; \kappa,c)$
in critical case. We always assume that $(\kappa, c), \kappa\neq0$, is critical in this section, except for Theorem \ref{simple} and Theorem \ref{maximal}.

There exists $n\in\N$ such that $h_{n}(\k,c)=q^{3-2n}\k+q^{2n-3}\k^{-1}-(q-q^{-1})^2c=0$. 
If there exists another integer $n'\in\N$ such that $h_{n'}(\k,c)=0$, we must have $q^{3-2n}\kappa=q^{2n'-3}\kappa^{-1}$, that is, $\kappa=\pm q^{n+n'-3}$. 
Since $q$ is generic,
there are at most two positive integers $n$ such that $h_n(\k,c)=0$.
For convenience, we set $$n_+=\max\{n\ |\ h_{n}(\k,c)=0\}\ \text{and}\ n_-=\min\{n\ |\ h_{n}(\k,c)=0\}.$$ 
We have seen $\kappa=\pm q^{n_++n_--3}$ if $n_+\neq n_-$.
Denote $u_\pm=u(n_\pm,0,1)$, where $1\in\C[K^{\pm1},C_1]$ is the constant polynomial. Recall that $u(n_\e,l,1)=q^{2n_\e l}K_2^lu_\e$ for $\e\in\{+,-\}$. 

For any $w\in M(\eta)$, we still use $\bar w$ to denote its image in $V(\eta;\kappa,c)=M(\eta)/W(\eta;\k,c)$. 
Note that $J(\k,c)v_\eta\subseteq\C[K^{\pm1},C_1]v_\eta\subseteq M_{0,0,0}$, so there is a natural filtration on $V(\eta;\k,c)$ inherited from that on $M(\eta)$ (see Section 3) and we can define the degree of an element in $V(\eta;\k,c)$. Namely, $\deg(\bar w)=\deg(w)$ for any $w\in M(\eta)$. 
In particular, 
$$\deg(\bar u(n_\e,l,1))=(n_\e,n_\e,l),\ \forall\ \e\in\{+,-\}.$$

Fix $\e\in\{+,-\}$. By Proposition \ref{E_2 action}, $\bar u_\e$ is a Whittaker vector in $V(\eta;\k,c)$ of type $(\a,0)$ and hence generates a proper submodule of $V(\eta;\k,c)$. 
We denote this submodule by $\tildeW_\e(\eta;\k,c)$ and
its preimage in $M(\eta)$ by $\hatW_\e(\eta;\k,c)$. 
More precisely, 
$$\hatW_\e(\eta;\k,c)=Uu_\e+W(\eta;\k,c)$$
and 
$$\tildeW_\e(\eta;\k,c)=U\bar u_\e=\hatW_\e(\eta;\k,c)/W(\eta;\k,c)\subseteq V(\eta;\k,c).$$
Note that the least degree of nonzero elements in $\tildeW_\e(\eta;\k,c)$ is $(n_\e,n_\e,0)=\deg(\bar u_\e)$.
%

\begin{lemma}\label{F_1 action} For any $n\in\Z_+, l\in\Z$ and $Q\in\C[K^{\pm1},C_1]$, we have
$$\aligned
F_1 u(n,l,Q)
=&\a^{-1}(q-q^{-1})^{-2}\Big(u(n,l,q^{-l}(q^{n-3}K^{-1}+q^{3-n}K)Q)\\
&\hskip1.5cm-u(n,l+2,q^{-n-l-1}K^{-1}Q)-u(n,l-2,q^{n-l+1}KQ)\\
&-q^{(2n-1)l}K_2^l\sum_{k=0}^{n}\sum_{j=0}^{n-k}q^{n-2k-2j}a_{kj}(n)F_2^{n-k}F_3^k K_2^{2j}K^{-2j-k}h_n(K,C_1)Q v_\eta\Big).
\endaligned$$
\end{lemma}

\begin{proof} Recall that $u(n,l,Q)=q^{2nl}K_2^lu(n,0,Q)$ and $F_1K_2=q^{-1}K_2F_1$. 
We need only to prove the lemma for $l=0$.

Applying $F_1$ to $u(n,0,Q)$ (see \eqref{u(n,l,Q)}), we get
$$\aligned
F_1 u(n,0,Q)=&\sum_{k=0}^{n}\sum_{j=0}^{n-k}a_{kj}(n)F_1F_2^{n-k}F_3^k K_2^{2j}K^{-2j-k}Q v_\eta\\
=&\sum_{k=0}^{n}\sum_{j=0}^{n-k}a_{kj}(n)(q^{n-k}F_2^{n-k}F_1+[n-k]F_2^{n-k-1}F_3)F_3^k K_2^{2j}K^{-2j-k}Q v_\eta\\
=&\sum_{k=0}^{n}\sum_{j=0}^{n-k}q^{n-2k-2j}a_{kj}(n)F_2^{n-k}F_3^k K_2^{2j}F_1K^{-2j-k}Q v_\eta\\
&+\sum_{k=0}^{n-1}\sum_{j=0}^{n-k}[n-k]a_{kj}(n)F_2^{n-k-1}F_3^{k+1} K_2^{2j}K^{-2j-k}Q v_\eta\\
=&\sum_{k=0}^{n}\sum_{j=0}^{n-k}\a^{-1}q^{n-2k-2j}a_{kj}(n)F_2^{n-k}F_3^k K_2^{2j}\Big(C_1-\frac{qKK_2^{-2}+q^{-1}K^{-1}K_2^2}{(q-q^{-1})^2}\Big)K^{-2j-k}Q v_\eta\\
&+\sum_{k=1}^{n}\sum_{j=0}^{n-k+1}[n-k+1]a_{k-1,j}(n)F_2^{n-k}F_3^{k} K_2^{2j}K^{-2j-k+1}Q v_\eta,\\
\endaligned$$
where we used the identity $F_1F_2^m-q^mF_2^mF_1=[m]F_2^{m-1}F_3$ for any $m\in\Z_+$.

Substituting the formula $C_1=(q-q^{-1})^{-2}\big(q^{3-2n}K+q^{2n-3}K^{-1}-h_n(K,C_1)\big)$ into the above equation, we obtain
$$\aligned
&\a(q-q^{-1})^{2}F_1 u(n,0,Q)+u(n,2,q^{-n-1}K^{-1}Q)+u(n,-2,q^{n+1}KQ)\\
=&\sum_{k=0}^{n}\sum_{j=0}^{n-k}(q-q^{-1})^{2}q^{n-2k-2j}a_{kj}(n)F_2^{n-k}F_3^k K_2^{2j}C_1K^{-2j-k}Q v_\eta\\
&-\sum_{k=0}^{n}\sum_{j=0}^{n-k}(q^{n-2k-2j+1}-q^{n+1-2k})a_{kj}(n)F_2^{n-k}F_3^k K_2^{2j-2}K^{1-2j-k}Q v_\eta\\
&-\sum_{k=0}^{n}\sum_{j=0}^{n-k}(q^{n-2k-2j-1}-q^{-n-1+2k})a_{kj}(n)F_2^{n-k}F_3^k K_2^{2j+2}K^{-2j-k-1}Q v_\eta\\
&+\sum_{k=0}^{n}\sum_{j=0}^{n-k+1}\a(q-q^{-1})^{2}[n-k+1]a_{k-1,j}(n)F_2^{n-k}F_3^{k} K_2^{2j}K^{-2j-k+1}Q v_\eta\\
\endaligned$$$$\aligned
=&-\sum_{k=0}^{n}\sum_{j=0}^{n-k}q^{n-2k-2j}a_{kj}(n)F_2^{n-k}F_3^k K_2^{2j}h_n(K,C_1)K^{-2j-k}Q v_\eta\\
&+\sum_{k=0}^{n}\sum_{j=0}^{n-k}q^{n-2k-2j}a_{kj}(n)F_2^{n-k}F_3^k K_2^{2j}(q^{3-2n}K+q^{2n-3}K^{-1})K^{-2j-k}Q v_\eta\\
&-\sum_{k=0}^{n}\sum_{j=-1}^{n-k-1}(q^{n-2k-2j-1}-q^{n+1-2k})a_{k,j+1}(n)F_2^{n-k}F_3^k K_2^{2j}K^{-1-2j-k}Q v_\eta\\
&-\sum_{k=0}^{n}\sum_{j=1}^{n-k+1}(q^{n-2k-2j+1}-q^{-n-1+2k})a_{k,j-1}(n)F_2^{n-k}F_3^k K_2^{2j}K^{-2j-k+1}Q v_\eta\\
&+\sum_{k=0}^{n}\sum_{j=0}^{n-k+1}\a(q-q^{-1})^{2}[n-k+1]a_{k-1,j}(n)F_2^{n-k}F_3^{k} K_2^{2j}K^{-2j-k+1}Q v_\eta\\
=&\sum_{k=0}^{n}\sum_{j=-1}^{n-k+1}F_2^{n-k}F_3^k K_2^{2j}f_{kj}(K,C_1)K^{-2j-k}Q v_\eta\\
&-\sum_{k=0}^{n}\sum_{j=0}^{n-k}q^{n-2k-2j}a_{kj}(n)F_2^{n-k}F_3^k K_2^{2j}h_n(K,C_1)K^{-2j-k}Q v_\eta\\
\endaligned$$
where $a_{kj}(n)=0$ whenever $k<0,j<0$ or $k+j>n$ and
$$\aligned
f_{kj}(K,C_1)=&a_{kj}(n)q^{n-2k-2j}(q^{3-2n}K+q^{2n-3}K^{-1})-(q^{n-2k-2j-1}-q^{n+1-2k})a_{k,j+1}(n)K^{-1}\\
&-(q^{n-2k-2j+1}-q^{-n-1+2k})a_{k,j-1}(n)K+\a(q-q^{-1})^{2}[n-k+1]a_{k-1,j}(n)K\\
=&a_{kj}(n)\Big(q^{3n-2k-2j-3}K^{-1}-q^{2n-k-j-3}(q-q^{-1})[n-k-j]K^{-1}\\
&q^{3-n-2k-2j}K+q^{3-n-k-j}(q-q^{-1})\frac{[n-k+1][k]+[n-2k-j+1][j]}{[n-k-j+1]}K\Big)\\
=&a_{kj}(n)(q^{n-3}K^{-1}+q^{3-n}K).\\
\endaligned$$
Combining the above two formulas and noticing $f_{kj}(K,C_1)=0$ for $j=0,n-k+1$, we obtain 
$$\aligned
&\a(q-q^{-1})^{2}F_1 u(n,0,Q)+u(n,2,q^{-n-1}K^{-1}Q)+u(n,-2,q^{n+1}KQ)\\
=&u(n,0,(q^{n-3}K^{-1}+q^{3-n}K)Q)-\sum_{k=0}^{n}\sum_{j=0}^{n-k}q^{n-2k-2j}a_{kj}(n)F_2^{n-k}F_3^k K_2^{2j}K^{-2j-k}h_n(K,C_1)Q v_\eta,\\
\endaligned$$
which is just the desired result for $l=0$.
\end{proof}

\begin{corollary}\label{C_1 act} 
For any $n\in\Z_+, l\in\Z$ and $Q\in\C[C_1,K]$, we have
$$\aligned
C_1 u(n,l,Q)
=&(q-q^{-1})^{-2}\Big(u(n,l,(q^{n-3}K^{-1}+q^{3-n}K)Q)\\
&-q^{2nl}K_2^l\sum_{k=0}^{n}\sum_{j=0}^{n-k}q^{n-2k-2j}a_{kj}(n)F_2^{n-k}F_3^k K_2^{2j}K^{-2j-k}h_n(K,C_1)Q v_\eta\Big).
\endaligned$$
\end{corollary}

\begin{proof} Follows from the definition of $C_1$, Lemma \ref{F_1 action} and the fact 
$$\hskip0.5cm E_1u(n,l,Q)=\a q^lu(n,l,Q),\ \ K_1u(n,l,Q)=K_2^{-2}Ku(n,l,Q)=q^{n}u(n,l-2,KQ).\hskip0.5cm\qedhere$$
\end{proof}

For convenience, we denote 
$$\k_\e=q^{-3n_\e}\k,\quad c_\e=\frac{q^{n_\e-3}\k^{-1}+q^{3-n_\e}\k}{(q-q^{-1})^2},\ \forall\ \e\in\{+,-\}.$$

\begin{corollary}\label{C_1 action} 
In the module $V(\eta;\k,c)$, we have $K\bar u_\e=\k_\e \bar u_\e, C_1\bar u_\e=c_\e\bar u_\e$ for $\e=+,-$.
\end{corollary}

\begin{proof} The first equation follows from 
$Ku(n,l,Q)=q^{-3n}u(n,l,KQ)$ and the second one follows from
the fact $h_{n_\e}(\k,c)=0$ and Corollary \ref{C_1 act}.
\end{proof}

\begin{corollary}\label{EndV} For any $(\k',c')\in\C^2, \k'\neq0$, we have another quotient module $V(\eta;\k',c')$.
\begin{itemize}
\item[(a)] $\Hom_U(V(\eta;\k',c'),V(\eta;\k,c))=0$ if $(\k',c')\notin\{(\k,c), (\k_+, c_+), (\k_-, c_-)\}$.
\item[(b)] $\End_U(V(\eta;\k,c))=\C\id$ and $\Hom_U(V(\eta;\k_\e,c_\e),V(\eta;\k,c))=\C\phi_\e$, where $\id$ is the identity map, $\phi_\e$ is the $U$-module homomorphism defined by $\phi_\e(\bar v'_\eta)=\bar u_\e$ for $\e=+,-$ and $\bar v'_\eta$ is the image of $v_\eta$ in $V(\eta;\k_\e,c_\e)$.
\item[(c)] In particular, $V(\eta;\k,c)$ is indecomposable.
\end{itemize}
\end{corollary}

\begin{proof} 
By Theorem \ref{vector}, the set of Whittaker vectors of type $(\a, 0)$ in $V(\eta;\k,c)$ is 
spanned by $\bar v_\eta$ and $\bar u_\pm=\bar u(n_\pm,0,1)$, which are eigenvectors of distinct eigenvalues 
$(\k,c), (\k_\pm,c_\pm)$ respectively, with respect to the action of $(K,C_1)$,
by Corollary \ref{C_1 action}

On the other hand, $\bar v'_\eta\in V(\eta;\k',c')$ is an eigenvector of eigenvalue $(\k',c')$ with respect to the action of $(K,C_1)$. Any $\phi\in\Hom_U(V(\eta;\k',c'),V(\eta;\k,c))$ must take $\bar v'_\eta$ to a multiple of one of $\bar v_\eta,\bar u_\pm$.
The results follow since $V(\eta;\k',c')$ is generated by $\bar v'_\eta$.
%
\end{proof}

\begin{corollary}\label{tildeW}
The module $\tildeW_\e(\eta;\k,c)$ has a basis $\{F_2^jF_3^kK_2^l\bar u_\e\ |\ j,k\in\Z_+, l\in\Z\}$ for $\e=+,-$. Moreover, we have the $U$-module isomorphism:
$$\tildeW_\e(\eta;\k,c)\cong V(\eta;\k_\e,c_\e),\ \forall\ \e\in\{+,-\}.$$
\end{corollary}

\begin{proof} It is obvious that $\{F_2^jF_3^kK_2^l\bar v_\eta\ |\ j,k\in\Z_+, l\in\Z\}$ is a basis of $V(\eta;\k,c)$. On the other hand, 
we have $\tildeW_\e(\eta;\k,c)=\spn\{F_2^jF_3^kK_2^l\bar u_\e\ |\ j,k\in\Z_+, l\in\Z\}$, which is a linearly-independent subset since elements therein have distinct degrees.
The first result follows.

The $U$-module homomorphism $\phi_\e$ in Corollary \ref{EndV} (b) is an isomorphism 
from $V(\eta;\k_\e,c_\e)$ to $\tildeW_\e(\eta;\k,c)$,
since it sends basis to basis, proving the second result.
%
%
\end{proof}

By Corollary \ref{tildeW}, we see that $\bigoplus_{(j,k,l)\in\Gamma}\C F_2^jF_3^kK_2^lu(n_\e,0,1)$ is a 
complemental subspace of $W(\eta;\k,c)$ in $\hatW_\e(\eta; \k,c)$, that is, as a vector space,
\begin{equation}\label{hat W}\aligned
\hatW_\e(\eta; \k,c)=&W(\eta; \k,c)\oplus\Big(\bigoplus_{(j,k,l)\in\Gamma}\C F_2^jF_3^kK_2^lu_\e\Big)\\
=&\Big(\bigoplus_{(j,k,l)\in\Gamma}\C F_2^jF_3^kK_2^lJ(\eta;\k,c)v_\eta\Big)\oplus\Big(\bigoplus_{(j,k,l)\in\Gamma}\C F_2^jF_3^kK_2^lu(n_\e,0,1)\Big)\\
=&\Big(\bigoplus_{(j,k,l)\in\Gamma}\C F_2^jF_3^kK_2^lJ(\eta;\k,c)v_\eta\Big)\oplus\Big(\bigoplus_{(j,k,l)\in\Gamma}\C F_2^jF_3^ku(n_\e,l,1)\Big).
\endaligned\end{equation}

%

Set $\hat\Gamma_\e=\{(j,k,l)\in\Gamma, 0\leq k< n_\e\}$ for $\e=+,-$. 
Since $\deg u(n_\e,0,1)=(n_\e,n_\e,0)$, 
we also deduce that $\sum_{(j,k,l)\in\hat\Gamma_\e}F_2^jF_3^kK_2^l\bar v_\eta$ is a complemental subspace of 
$\tildeW_\e(\eta;\k,c)$ in $V(\eta;\k,c)$ from Corollary \ref{tildeW}. 
Hence $\sum_{(j,k,l)\in\hat\Gamma_\e}F_2^jF_3^kK_2^l v_\eta$ 
is a complemental subspace of $\hatW_\e(\eta;\k,c)$ in $M(\eta)$, that is,
\begin{equation}\label{hat Gamma}
M(\eta)=\hatW_\e(\eta;\k,c)\oplus \bigoplus_{(j,k,l)\in\hat\Gamma_\e}F_2^jF_3^kK_2^lv_\eta,
\end{equation}

%
%
Then we can deduce the irreducibility of $\tildeW_\e$ for $\e=+,-$. Recall that $n_-\leq n_+$.

\begin{theorem}\label{critical} Let notation be as before. 
\begin{itemize}
\item[(a)] $\tildeW_-(\eta;\k,c)$ is a maximal submodule of $V(\eta;\k,c)$.
\item[(b)] $\tildeW_+(\eta;\k,c)$ is a minimal submodule of $V(\eta;\k,c)$.
\item[(c)] If $n_+=n_-$, then $\tildeW_+(\eta;\k,c)$ is the unique nonzero proper submodule of $V(\eta;\k,c)$. In particular, $V(\eta;\k,c)$ admits a unique composition series
$$0\subset\tildeW_+(\eta;\k,c)=\tildeW_-(\eta;\k,c)\subset V(\eta;\k,c).$$
\item[(d)] If $n_-<n_+$, then $\tildeW_+(\eta;\k,c)\subset \tildeW_-(\eta;\k,c)$ are the only two nonzero proper submodules of $V(\eta;\k,c)$. In particular, $V(\eta;\k,c)$ admits a unique composition series
$$0\subset\tildeW_+(\eta;\k,c)\subset \tildeW_-(\eta;\k,c)\subset V(\eta;\k,c).$$
\end{itemize}
\end{theorem}

\begin{proof} 
(a) Fix an $\e\in\{+,-\}$. Let $N$ be a proper submodule of $M(\eta)$ that properly contains $\hatW_\e=\hatW_\e(\eta;\k,c)$. 
Take any $w\in N\setminus \hatW_\e$ with minimal degree, then 
by \eqref{hat W} and \eqref{hat Gamma}, we can write 
$$w=\sum_{(j,k,l)\in\hat\Gamma_\e}a_{j,k,l}F_2^jF_3^kK_2^lv_\eta,$$
where only finitely many $a_{j,k,l}\in\C$ are nonzero.

Set $\deg(w)=(j_0,k_0,l_0)\in\hat\Gamma_\e$ and $n_0=j_0+k_0$. If $n_0=0$, then $w=\sum_{l\in\Z}a_{0,0,l}K_2^lv_\eta$, forcing $v_\eta\in N$ by Lemma \ref{claim'} and hence $N=M(\eta)$, contradiction.  
So we have $n_0\geq 1$.

By Lemma \ref{compute} (4), we have $\deg((E_1-\a q^{l_0})w)<\deg w$ and $\deg(E_2w)<\deg w$, hence $(E_1-\a q^{l_0})w, E_2w\in\hatW_\e$. 

On the other hand, both $(E_1-\a q^{l_0})w$ and $E_2w$ are linear combinations of elements in $F_2^jF_3^kK_2^l\C[K^{\pm1},C_1]v_\eta, 0\leq k<n_\e$ by Lemma \ref{compute}, c.f. \eqref{E_1-alpha} and \eqref{E_2u}.
By \eqref{hat W}, we must have $(E_1-\a q^{l_0})w, E_2w\in W(\eta; \k,c)$.

Write $w=\sum_{i=0}^{p}w_{n_i}$ with $0\neq w_{n_i}\in M_{n_i}$ for distinct $n_i\in\Z_+$ and $p\in\Z_+$. 
By Lemma \ref{claim} and Lemma \ref{compute} (4), we also have 
$(E_1-\a q^{l_0})w_{n_i}, E_2w_{n_i}\in W(\eta; \k,c)$, forcing each $w_{n_i}$ to be a multiple of one
of the elements $u(n_\pm,l_0,1)\in M_{n_{\pm}}$ or $K_2^{l_0}v_\eta\in M_0$. Hence $n_i\in\{0,n_+,n_-\}$. Note that $n_i<n_\e$ for $i=0,1,\dots,p$.

Now take $\e=-$, then we have $p=0$ and $n_0=0$, which indicate that $w$ is a nonzero multiple of $v_\eta$, contradiction. So $\hatW_-(\eta;\k,c)$ is a maximal submodule of $M(\eta)$ and hence $\tildeW_-(\eta;\k,c)$ is a maximal submodule of $V(\eta;\k,c)$. This proves (a).
\smallskip

(b) 
By Lemma \ref{tildeW}, we have $\tildeW_\e(\eta;\k,c)\cong V(\eta;\k_\e,c_\e)$ as a $U$-module,
where $\k_\e=q^{-3n_\e}\k$, $c_\e=(q-q^{-1})^{-2}(q^{n_\e-3}\k^{-1}+q^{3-n_\e}\k)$.
If $V(\eta;\k_\e,c_\e)$ is not irreducible, then $(\k_\e,c_\e)$ is critical. 
Say, $h_{n'}(\k_\e,c_\e)=0$ for some $n'\in\N$, then $V(\eta;\k_\e,c_\e)$ has a Whittaker vector of 
degree $(n',n',0)$ and hence $\tildeW_\e(\eta;\k,c)$ has a Whittaker vector of 
degree $(n'+n_\e,n'+n_\e,0)$, according to the isomorphism $\phi_\e$ in Corollary \ref{tildeW} and \ref{EndV}). Since the Whittaker vector in $\tildeW_\e(\eta;\k,c)$ is also a Whittaker vector in $V(\eta;\k,c)$, we must have $n_\e=n_-$ and $n'+n_\e=n_+$.
If $\e=+$, this is impossible and hence $\tildeW_+(\eta;\k,c)$ is irreducible, proving (b).

\smallskip
(c) follows from a combination of (a), (b) and the fact that $V(\eta;\k,c)$ is indecomposable by Corollary \ref{EndV}.

\smallskip
(d) Continuing the argument in (b) and taking $\e=-$. 
Recall that $n_-<n_+$ implies $\k=\pm q^{n_++n_--3}$ and
$$\k_-=\pm q^{n_+-2n_--3},\quad c_-=\pm(q-q^{-1})^{-2}(q^{-n_+}+q^{n_+}).$$ 
It is easy to check that $n'=n_+-n_-$ is the only integer such that $h_{n'}(\k_-,c_-)=0$.
%
%
%
%
Then $\tildeW_-(\eta;\k,c)\cong V(\eta;\k_-,c_-)$ has a unique nonzero proper submodule by (c). 
%

If $\tildeW_+(\eta;\k,c)\not\subseteq \tildeW_-(\eta;\k,c)$, then 
the minimality of $\tildeW_+(\eta;\k,c)$ and maximality of $\tildeW_-(\eta;\k,c)=0$ imply 
$\tildeW_+(\eta;\k,c)\cap \tildeW_-(\eta;\k,c)=0$ and
$$V(\eta;\k,v)=\tildeW_+(\eta;\k,c)\oplus \tildeW_-(\eta;\k,c),$$
contradicting the fact that $V(\eta;\k,c)$ is indecomposable by Corollary \ref{EndV}.
So $\tildeW_+$ is the only nonzero proper submodule of $\tildeW_-$. 

Now let $\bar N\subseteq V(\eta;\k,c)$ be an arbitrary nonzero proper submodule other than $\tildeW_\pm$. 
If $\bar N\cap\tildeW_-=0$, then $V(\eta;\k,c)=\bar N\oplus\tildeW_-$, contradiction.
So $\bar N\cap\tildeW_-=\tildeW_+$ and hence $\tildeW_+$ is a proper submodule of $\bar N$. 
By the argument in (a) for $\e=+$, we have $\bar w=aK_2^{l_0}\bar v_\eta+\bar u(n_-,l_0,1)\in \bar N$ for some $a\in\C$ and $l_0\in\Z$. Since $K_2^{l_0}\bar v_\eta$ and $\bar u(n_-,l_0,1)$ are eigenvectors of different eigenvalues with respect to the action of $K$ and $C_1$, we see $\bar u(n_-,l_0,1)\in\bar N$, forcing $\bar u_-\in \bar N$ and $\tildeW_-\subseteq\bar N$, contradiction. (d) is proved.
%
\end{proof}

By the proof of the above theorem, we can determine the Whittaker vectors of the quotient modules of $V(\eta;\k,c)$.

\begin{corollary} \label{vector'}
Let $l\in\Z$, $\bar v_\eta$ be the corresponding image of $v_\eta$ in corresponding quotient modules of $M(\eta)$ and other notation be as before.
\begin{itemize}
\item[(a)] If $(\k,c)$ is non-critical, the set of all Whittaker vectors of type $(\a q^l,0)$
in $V(\eta;\k,c)$ is $\spn\{K_2^l\bar v_\eta\}$.
\item[(b)] If $(\k,c)$ is critical, the set of all Whittaker vectors of type $(\a q^l,0)$
in $V(\eta;\k,c)$, $V(\eta;\k,c)/\tildeW_+$ and $V(\eta;\k,c)/\tildeW_-$ are respectively
$$\spn\{K_2^l\bar v_\eta, \bar u(n_\pm,l,1)\},\quad \spn\{K_2^l\bar v_\eta, \bar u(n_-,l,1)\},\quad 
\spn\{K_2^l\bar v_\eta\}.$$
\end{itemize}
\end{corollary}

\begin{proof} Part (a) and the first assertion of (b) follow from Theorem \ref{vector} and 
other assertions of (b) follow from the proof of Theorem \ref{critical} (a). 
\end{proof}

\begin{remark}\label{g_n}
Let notation be as in Theorem \ref{critical}. Suppose $n_-<n_+$.
%
%
If we denote
$$g_n=g_n(F_2,F_3,K_2,K)=\sum_{k=0}^{n}\sum_{j=0}^{n-k}a_{kj}(n)F_2^{n-k}F_3^kK_2^{2j}K^{-2j-k},\ \forall\ n\in\N,$$
then $u(n,0,Q)=g_nQv_\eta$ for any $Q\in\C[K,C_1]$. In particular, $u_\pm=g_{n_\pm}v_\eta$. 

By the proof of Theorem \ref{critical}, $V(\eta;\k_-,c_-)$ has only one Whittaker vector of type $(\a,0)$ up to a scalar, 
which is just the image of $u(n_+-n_-,0,1)=g_{n_+-n_-}v_\eta\in M(\eta)$ in $V(\eta;\k_-,c_-)$ and corresponds to 
$g_{n_+-n_-}\bar u_-$ under the isomorphism $V(\eta;\k_-,c_-)\cong\tildeW_-$ (see Corollary \ref{tildeW}). Then we have $g_{n_+-n_-}\bar u_-=a\bar u_+$ for some $a\in\C$. 
Comparing the coefficients, we get $a=1$, which is equivalent to 
$g_{n_+-n_-}g_{n_-}\bar v_\eta=g_{n_+}\bar v_\eta$, or, 
$g_{n_+-n_-}g_{n_-}=g_{n_+}$. This motivates us to deduce the following simpler formula of $u(n,l,Q)$.
\end{remark}

\begin{corollary}\label{u(n)}
For any $n\in\N, l\in\Z, Q\in\C[K^{\pm1},C_1]$, we have $u(n,l,Q)=q^{2nl}K_2^{l}g^nQv_\eta$, where
$$g=g_1(F_2,F_3,K_2,K)=F_2(1-q^{-2}K_2^{2}K^{-2})+\a(1-q^{-2})F_3K^{-1}.$$
\end{corollary}

\begin{proof} 
By the definition of $u(n,l,Q)$, we need only to prove $g_n=g_1^n$, 
or equivalently, $g_{n+1}=g_{n}g_1$ for any $n\in\N$. This can be done straightforward by induction.
\end{proof}
%

Finally, we conclude this paper with a characterization of simple Whittaker modules over $U_q(\sl_3)$.
Let us first recall the center of $U_q(\sl_3)$ from \cite{LXZ}. Set 
$$\tilde F_3=F_2F_1-qF_1F_2\ \text{and}\ \tilde E_3=E_2E_1-q^{-1}E_1E_2.$$ 
Then the center of $U_q(\sl_3)$ is generated by $X_1X_2, X_1^3K, X_2^3K^{-1}$, where
$$X_1=q^{-3}K^{-1}+qK_1+q^{-1}K_1^{-1}+(q-q^{-1})^2\big(F_1E_1+q^{-2}F_2E_2K_1^{-1}K_2^{-1}-q^{-1}F_{3}E_{3}K_2^{-1}\big),$$
$$X_2=q^{3}K+qK_1+q^{-1}K_1^{-1}+(q-q^{-1})^2\big(F_1E_1+q^2F_2E_2K_1K_2-q\tilde F_3\tilde E_3K_2\big).$$

\begin{theorem}\label{simple}
Any irreducible Whittaker module of type $\eta$ is isomorphic to one of:
\begin{itemize}
\item[(a)] $V(\eta;\k,c)=M(\eta)/W(\eta;\k,c)$ with $(\k,c)$ non-critical;
\item[(b)] $V(\eta;\k,c)/V(\eta;\k_-,c_-)\cong V(\eta;\k,c)/\tildeW_-(\eta;\k,c)\cong M(\eta)/\hatW_-(\eta;\k,c)$ with $(\k,c)$ critical.
\end{itemize}
\end{theorem}

\begin{proof} 
Let $V$ be an irreducible Whittaker $U$-module of type $\eta$. 
Then $X_1X_2, X_1^3K, X_2^3K^{-1}$ act as scalars on $V$, say,
$x, a^3, b^3\in\C$. Then $x^3=a^3b^3$, so we may assume $x=ab$ by adjusting the value of $a$ with a cubic root multiple. 

Let $w$ be a nonzero type-$\eta$ Whittaker vector of $V$ and noticing $E_2w=0$, we have
\begin{equation}\label{X}\aligned
& X_1X_2w=((q-q^{-1})^2C_1+q^{-3}K^{-1})((q-q^{-1})^2C_1+q^{3}K)w=abw, \\
& X_1^3Kw=K((q-q^{-1})^2C_1+q^{-3}K^{-1})^3w=a^3w,\\
& X_2^3K^{-1}w=K^{-1}((q-q^{-1})^2C_1+q^{3}K)^3w=b^3w,
\endaligned\end{equation}
Combining these expressions, we can deduce
$$(q^{18}K^6+(3ab-3-q^{-3}b^3)q^{12}K^4+(q^3a^3+3-3ab)q^{6}K^2-1)w=0.$$
Then there exist some nonzero $\k\in\C$ and polynomial $f_1(K)\in\C[K]$ such that $f_1(K)w\neq0$ and $(K-\k)f_1(K)w=0$. 

On the other hand, by \eqref{X} we also have 
$$\k((q-q^{-1})^2C_1+q^{-3}\k^{-1})^3f_1(K)w=K((q-q^{-1})^2C_1+q^{-3}K^{-1})^3f_1(K)w=a^3f_1(K)w.$$
Similarly, we can find $c\in\C$ and polynomial $f_2(C_1)\in\C[C_1]$ such that $f_2(C_1)f_1(K)w\neq0$ and $(C_1-c)f_2(C_1)f_1(K)w=0$.

Since $w'=f_2(C_1)f_1(K)w$ is still a nonzero Whittaker vector of type $\eta$, there exists 
a surjective $U$-module homomorphism $\phi: M(\eta)\rightarrow V$ such that $\phi(v_\eta)=w'$.
It is clear that $W(\eta;\k,c)\subseteq \ker\phi$. 
So $\phi$ induce a nonzero $U$-module homomorphism $\bar\phi: V(\eta;\k,c)\rightarrow V$. 
As a result, $V\cong V(\eta;\k,c)$ if $(\k,c)$ is non-critical and 
$V\cong V(\eta;\k,c)/\tildeW_-(\eta;\k,c)$ if $(\k,c)$ is critical by Theorem \ref{non-critical} and Theorem 
\ref{critical} respectively. 
\end{proof}

\begin{theorem}\label{maximal}
Any maximal submodule of $M(\eta)$ is one of:
\begin{itemize}
\item[(a)] $W(\eta;\k,c)$ with $(\k,c)$ non-critical;
\item[(b)] $\hatW_-(\eta;\k,c)$ with $(\k,c)$ critical.
\end{itemize}
\end{theorem}

\begin{proof} Let $N$ be a maximal submodule of $M(\eta)$ and $I=I(N)$. 
Obviously, $v_\eta\notin N$ and $\bar v_\eta\in M(\eta)/N$, the image of $v_\eta\in M(\eta)$, 
is a nonzero Whittaker vector of type $\eta$.

By Theorem \ref{simple}, there exists a $U$-module isomorphism $\phi: M(\eta)/N\rightarrow V$, 
where $V=V(\eta;\k,c)$ for some non-critical pair $(\k,c)$ 
or $V=V(\eta;\k,c)/\tildeW_-(\eta;\k,c)$ for some critical pair $(\k,c)$.

Note that $\phi(\bar v_\eta)$ is a Whittaker vector of type $(\a,0)$ in $V$,
and hence a scalar multiple of the image of $v_\eta$ in $V$ by Corollary \ref{vector'}.
Hence $J(\k,c)\phi(\bar v_\eta)=0$ and also $J(\k,c)\bar v_\eta=0$. 
We have $W(\eta;\k,c)\subseteq N$ and the results follow from Theorem \ref{non-critical}
and Theorem \ref{critical}.
%
%
%
%
\end{proof}

\begin{remark}
Note that the above results are independent of $\eta$, that is, maximal submodules of a universal Whittaker module and irreducible Whittaker modules with respect to any Whittaker function can be described in the same way.
\end{remark}

\section*{Acknowledgement}
X.G. is partially supported by the NSF of China (Grant 11971440).
L.X. is partially supported by the NSF of China (Nos. 11871249, 11771142) and the Jiangsu Natural Science Foundation (No. BK20171294). The authors would like to thank Prof. Jianrong Li for valuable discussions 
at the beginning of the paper.

\end{document}